\documentclass[11pt]{amsart}
\usepackage{amssymb, amsfonts, amsthm, amsmath, mathrsfs, todonotes, tikz-cd, easytable,lineno}
\usepackage[linktocpage=true,hidelinks]{hyperref}
\usepackage{caption}
\usepackage{graphicx}
\usepackage{subcaption}
\usepackage[capitalise]{cleveref}
\usepackage{comment}
\usepackage[T1]{fontenc}
\setuptodonotes{inline}

\usepackage{mathtools}
\usepackage{extarrows}
\usepackage[numbers,sort&compress]{natbib}
\usepackage{xparse}
\usepackage{trimclip}
\makeatletter
\DeclareRobustCommand{\subto}{%
  \mathrel{\mathpalette\short@to\relax}%
}

\newcommand{\short@to}[2]{%
  \mkern2mu
  \clipbox{{.35\width} 0 0 0}{$\m@th#1\vphantom{+}{\rightarrow}$}%
  }
\makeatother
\DeclarePairedDelimiterX{\setbuilder}[1]{\{}{\}}{\setargs{#1}}
\NewDocumentCommand{\setargs}{>{\SplitArgument{1}{;}}m}
{\setargsaux#1}
\NewDocumentCommand{\setargsaux}{mm}
{\IfNoValueTF{#2}{#1} {#1\nonscript\:\delimsize\vert\allowbreak\nonscript\:\mathopen{}#2}}%

\newtheorem*{rep@theorem}{\rep@title}
\newcommand{\newreptheorem}[2]{%
\newenvironment{rep#1}[1]{%
 \def\rep@title{#2 \ref{##1}}%
 \begin{rep@theorem}}%
 {\end{rep@theorem}}}
\makeatother

\newtheorem*{rep@corollary}{\rep@title}
\newcommand{\newrepcorollary}[2]{%
\newenvironment{rep#1}[1]{%
 \def\rep@title{#2 \ref{##1}}%
 \begin{rep@corollary}}%
 {\end{rep@corollary}}}
\makeatother

\numberwithin{equation}{section}
\newtheorem{theorem}{Theorem}[section]
\newreptheorem{theorem}{Theorem}

\newtheorem{lemma}[theorem]{Lemma}
\newtheorem{proposition}[theorem]{Proposition}
\newtheorem{corollary}[theorem]{Corollary}
\newreptheorem{corollary}{Corollary}

\theoremstyle{definition}
\newtheorem{definition}[theorem]{Definition}
\newtheorem{remark}[theorem]{Remark}

\newtheorem{notation}[theorem]{Notation}
\newreptheorem{notation}{Notation}

\newcommand{\R}{\mathbb R}

\newcommand{\N}{\mathbb N}
\newcommand{\cat}[1]{#1}

\newcommand{\rips}{\mathcal{R}}
\newcommand{\srips}{\mathcal{SR}}

\DeclareMathOperator{\op}{op}

\newcommand{\gr}[1]{\text{gr}(#1)}

\newcommand{\Bary}[1]{{#1}^{+}}

\newcommand{\Poset}{{\mathbb N}^{\mathrm{op}}\times [0, \infty)}

\newcommand{\Simp}{\mathbf{Simp}}
\newcommand{\Top}{\mathbf{Top}}

\renewcommand{\Vec}{\mathbf{Vec}}
\newcommand{\Vect}{\mathbf{Vec}}
\newcommand{\eps}{\ensuremath{\epsilon}}

\newcommand{\subdiv}{\mathcal{SF}}

\DeclareMathOperator{\colim}{colim}

\DeclareMathOperator{\diam}{diam}
\DeclareMathOperator{\im}{im}
\newcommand{\Ch}[1]{C_{#1}\mkern 2mu}
\tikzset{%
  symbol/.style={
    draw=none,
    every to/.append style={
      edge node={node [sloped, allow upside down, auto=false]{$#1$}}
    },
  },
}

\begin{document}

\title[Sparse Approximation of the Subdivision-Rips Bifiltration]{Sparse Approximation of the Subdivision-Rips Bifiltration for Doubling Metrics}
\date{\today}

\author{Michael Lesnick}
\address{Department of Mathematics and Statistics, University at Albany - SUNY, Albany, NY, USA}
\email{\href{mailto:mlesnick@albany.edu}{mlesnick@albany.edu}}

\author{Kenneth McCabe}
\address{Department of Mathematics, Northeastern University, Boston, MA, USA}
\email{\href{mailto:mccabe.ke@northeastern.edu}{mccabe.ke@northeastern.edu}}

\begin{abstract}
The Vietoris-Rips filtration, the standard filtration on metric data in topological data analysis, is notoriously sensitive to outliers.  
Sheehy's subdivision-Rips bifiltration $\srips(-)$ is a density-sensitive refinement that is robust to outliers in a strong sense, but whose 0-skeleton has exponential size. For $X$ a finite metric space of constant doubling dimension and fixed $\epsilon>0$, we construct a $(1+\epsilon)$-homotopy interleaving approximation of $\srips(X)$ whose $k$-skeleton has size $O(|X|^{k+2})$.  For $k\geq 1$ constant, the $k$-skeleton can be computed in time $O(|X|^{k+3})$.
\end{abstract}

\maketitle

\section{Introduction}
This paper is a companion to a recent paper by the authors \cite{lesnickNerveModels2024}.   Our primary aim is to resolve a conjecture from \cite{lesnickNerveModels2024} concerning approximations of Sheehy's \emph{subdivision-Rips bifiltration} \cite{sheehyMulticoverNerve2012a} for metric spaces of bounded doubling dimension.
\subsection{Context}
We begin with a brief discussion of the context for our work, referring the reader to \cite{botnanIntroductionMultiparameter2023,blumbergStability2Parameter2022,lesnickNerveModels2024} for additional context.  

The \emph{(Vietoris-)Rips filtration} $\rips(-)$ is the standard filtration on metric data in topological data analysis.  It is widely used in applications, usually via homology computations, and has been the subject of extensive theoretical work.  For $X$ a finite metric space, the $k$-skeleton of $\mathcal R(X)$ has size $O(|X|^{k+1})$.  

While $\mathcal R(X)$ is stable to small perturbations of the data \cite{chazal2009gromov,chazal2014persistence}, it is highly unstable to outliers \cite[Section $4$]{blumbergRobustStatistics2014}, and can be insensitive to variations in density \cite[Figure $2$]{lesnickInteractive15}.  One natural way to address these limitations is to instead construct a \emph{bifiltration}, treating distance and density as separate parameters \cite{carlssonTheoryMultidimensional2009}.  Multiparameter persistence, the subfield of TDA that works with such multiparameter filtrations and their homology, has become one of the most active areas of TDA in recent years, with substantial progress on several fronts; see \cite{botnanIntroductionMultiparameter2023} for a detailed introduction.  

There have been several proposals for density-sensitive bifiltrations on metric data, offering different trade-offs between size, computability, and robustness to outliers \cite{lesnickInteractive15,blumbergStability2Parameter2022,sheehyMulticoverNerve2012a,hellmerDensity24}.  Among these, the subdivision-Rips bifiltration $\srips(-)$, introduced in \cite{sheehyMulticoverNerve2012a}, is notable for satisfying a strong robustness property \cite[Theorem 1.6\,(iii) and Remark 2.16]{blumbergStability2Parameter2022}.  However, we showed in \cite{lesnickNerveModels2024} that for a large class of planar point sets $X$, if $\mathcal F$ is a functor valued in simplicial complexes that is weakly equivalent to $\srips(X)$, then the 0-skeleton of $\mathcal F$ has size exponential in $|X|$.  (See \cref{Sec:Weak-Equivalence} for the definition of weak equivalence.)  This implies that exact computations of $\srips(X)$ are out of reach, even up to homotopy, and raises the question of whether $\srips(X)$ can be approximated by an object of reasonable size.  

In this paper, we will work with a notion of $(1+\epsilon)$-approximation of bifiltrations defined in terms of  \emph{(multiplicative) homotopy interleavings}, following \cite{blumbergUniversalityHomotopy2023,lesnickNerveModels2024,buchetSparseHigher2023b}; see \cref{Def:Approximation}.  Here $\epsilon\geq 0$, and the smaller the value of $\epsilon$, the better the approximation.  Informally, a $(1+\eps)$-approximation to $\srips(X)$ preserves the robustness property of $\srips(X)$, up to error $\eps$.

Several approximations of $\srips(X)$ are known: The \emph{degree-Rips bifiltration} $\mathcal{DR}(X)$, a simple and well-studied density-sensitive bifiltration, is (after linear rescaling) a $\sqrt{3}$-approximation of $\srips(X)$ whose $k$-skeleton has size $O(|X|^{k+2})$ \cite{lesnickInteractive15,blumbergStability2Parameter2022}.  The 2-parameter persistence software packages RIVET \cite{lesnickInteractive15} and Persistable \cite{Scoccola2023,rolle2020stable} support computations with $\mathcal{DR}(X)$.  
  It was recently discovered that $\srips(X)$ in fact admits a $\sqrt{2}$-approximation whose $k$-skeleton also has size $O(|X|^{k+2})$; indeed, \cite{hellmerDensity24} and  \cite[Corollary 1.5\,(i)]{lesnickNerveModels2024} give two different but weakly equivalent (and closely related) constructions of such a $\sqrt{2}$-approximation. 
 
 For arbitrary finite metric spaces $X$, the approximation factor of $\sqrt{2}$ is optimal among constructions with polynomially-sized skeleta: We showed in \cite[Corollary 1.5\,(ii)]{lesnickNerveModels2024} that for fixed $c\in [1,\sqrt{2})$, the 0-skeleton of any $c$-approximation of $\srips(X)$ has worst-case exponential size.    
However, if one restricts attention to a smaller class of metric spaces, one can get tighter approximations of $\srips(X)$: Our result \cite[Corollary 1.8\,(ii)]{lesnickNerveModels2024} established that for fixed $\epsilon \in (0,1)$, $p\in [1,\infty]$, and finite $X\subset \R^d$ endowed with the $\ell_p$-metric, there exists a $(1+\epsilon)$-approximation to $\srips(X)$ whose $k$-skeleton has size polynomial in $|X|$.  But the degree of the polynomial in this size bound is quite large, and depends exponentially on $d$ and $\eps$.

\subsection{Contributions}
In  \cite[Conjecture 1.9]{lesnickNerveModels2024}, we conjectured that \cite[Corollary 1.8\,(ii)]{lesnickNerveModels2024} extends to metric spaces of bounded doubling dimension.  In the present paper, we resolve this conjecture, and in fact do so with a much tighter size bound than that of \cite[Corollary 1.8\,(ii)]{lesnickNerveModels2024}.  Our main result is the following:
\begin{theorem} \label{Theo:Main-Theorem}
For $X$ a finite metric space of constant doubling dimension and any fixed $\epsilon>0$, there exists a $(1+\eps)$-approximation $\mathcal {NA}(X)$ to $\srips(X)$ whose $k$-skeleton has size $O(|X|^{k+2})$. \end{theorem}
For $k\geq 1$ constant, we also give a straightforward algorithm to compute the $k$-skeleton of $\mathcal {NA}(X)$ in time $O(|X|^{k+3})$; see \cref{Sec:Computation}.

 \cref{Theo:Main-Theorem} tells us that for finite metric spaces $X$ of constant doubling dimension, $\srips(X)$ admits approximations to \emph{arbitrary accuracy} satisfying the same asymptotic size bounds as the $\sqrt{3}$- and $\sqrt{2}$-approximations of previous work.  However, the bound of \cref{Theo:Main-Theorem} hides a large constant that depends exponentially on $\eps$ and the doubling dimension.  As such, the path from our results to practical computations is not yet clear.  We leave the exploration of this to future work.  

The approximation $\mathcal NA(X)$ is generally not a bifiltration, but rather a \emph{semifiltration}, meaning that its structure maps are guaranteed to be inclusions in only one of the parameter directions (see \cref{Sec:Filtrations_Bifiltrations} below).  The approximations of \cite{lesnickNerveModels2024} mentioned above are also  semifiltrations.  However, as explained in \cite[Appendix A]{{lesnickNerveModels2024}}, a construction of Kerber and Schreiber \cite{kerber2019barcodes} extends to convert a simplicial semifiltration of bounded pointwise dimension to a bifiltration, with only logarithmic increase in size.  This, together with \cref{Theo:Main-Theorem}, implies the following (cf. \cite[Corollary 1.10]{lesnickNerveModels2024}):

\begin{corollary} \label{Cpr:Bifiltration}
For $X$ a finite metric space of constant doubling dimension and any fixed $\epsilon>0$, there exists a \emph{bifiltered} $(1+\eps)$-approximation to the $k$-skeleton of $\srips(X)$ that has size $O(|X|^{k+2}\log |X|)$.
\end{corollary}

\subsection{Proof Strategy}
The main step in our proof of  \cref{Theo:Main-Theorem} is to show that for $X$ a finite metric space of bounded doubling dimension, its Rips filtration $\rips(X)$ admits a $(1+\epsilon)$-interleaving approximation $\mathcal A(X)$ with $O(|X|)$ distinct simplicial complexes, each with $O(|X|)$ maximal simplices.  Letting $\mathcal{SA}(X)$ denote the subdivision bifiltration of $\mathcal A(X)$, the main result of \cite{lesnickNerveModels2024} (\cref{Theo:Nerve-Model} below) then yields a functor $\mathcal{NA}(X)$ valued in simplicial complexes that is weakly equivalent to $\mathcal{SA}(X)$ and has size $O(|X|^{k+2})$.  The functor $\mathcal{NA}(X)$ is defined via a nerve construction (see \cref{Sec:Nerve_Models}).  The interleaving between $\rips(X)$ and $\mathcal A(X)$ induces an interleaving between $\srips(X)$ and $\mathcal{SA}(X)$ (see \cref{Lem:Inerleaving_and_Subdivision}), implying that $\mathcal{NA}(X)$ is a $(1+\eps)$-homotopy interleaving approximation of $\srips(X)$.

To prove that  $\mathcal A(X)$ has $O(|X|)$ distinct simplicial complexes (\cref{Lem:Num_Distinct_Complexes}), we apply a result of Har-Peled and Mendel \cite[Lemma 5.1]{har-peledFastConstruction2005}, which says that doubling metrics have \emph{well-separated pair decompositions} of linear size (see \cref{Sec:WSPDs}).  In fact, one can prove both \cref{Theo:Main-Theorem} and the runtime bound on our main algorithm without using either \cite[Lemma 5.1]{har-peledFastConstruction2005} or \cref{Lem:Num_Distinct_Complexes}, but \cref{Lem:Num_Distinct_Complexes} enables simplifications of both proofs.

\subsection{Related Work on Linear Approximations of (Bi)filtrations}
The problem of approximating a 1-parameter filtration by one of linear size has been extensively studied by applied topologists, first in seminal work of Sheehy  on approximations of Rips filtrations \cite{sheehyLinearSizeApproximations2013a}, and subsequently in many other papers \cite{choudhary2021improved,brehm2018sparips,sheehy2021sparse,choudhary2019improved,botnan2015approximating,brun2019sparse,dey2019simba,dey2014computing,buchetEfficientRobust2016,cavanna2015geometric,choudhary2019polynomial}. 

 In the 2-parameter setting, Buchet, Dornelas, and Kerber have given a linear-size $(1+\epsilon)$-approximation of the \emph{multicover bifiltration} for any fixed $\epsilon>0$ \cite{buchetSparseHigher2023b}.  The multicover bifiltration is a density-sensitive bifiltration on a point cloud in Euclidean space satisfying a robustness property similar to that of $\mathcal{SR}(X)$ \cite[Theorem 1.6\,(i)]{blumbergStability2Parameter2022}.   However, the linear size bound of \cite{buchetSparseHigher2023b} requires that one of the bifiltration parameters is restricted to be less than some constant $\mu$; the given bound depends exponentially on $\mu$.  In comparison, our size bound on $\mathcal{NA}(X)$ applies to a larger class of data sets $X$ (namely, finite metric spaces of bounded doubling dimension) and requires no parameter thresholding.  On the other hand, our bound is not linear in $|X|$, and applies only to the fixed-dimensional skeleta of $\mathcal{NA}(X)$.

\subsection{Outline} \cref{Sec:Preliminaries} reviews preliminaries.  
  \cref{Sec:Proof_Main} gives the proof of \cref{Theo:Main-Theorem}.   \cref{Sec:Computation} presents our algorithm for computing $\mathcal{NA}(X)$.  

\section{Preliminaries}
\label{Sec:Preliminaries}

In this section, we cover the preliminaries needed for our results. \cref{Sec:Filtrations_Bifiltrations,sec:Interleavings} discuss (bi)filtrations and homotopy interleavings, closely following parts of \cite[Section 2]{lesnickNerveModels2024}.  \cref{Sec:Nerve_Models} reviews the nerve models of subdivision bifiltrations introduced in \cite{lesnickNerveModels2024}.  \cref{Sec:Doubling} introduces doubling dimension.  \cref{Sec:WSPDs} discusses well-separated pair decompositions.

\subsection{Filtrations and Bifiltrations}
\label{Sec:Filtrations_Bifiltrations}

We regard a poset $P$ as a category in the usual way, i.e., the set of objects is $P$ and morphisms are pairs $p\leq q$. For any category $\mathbf C$, functor $F:P\to \mathbf C$, and $p\leq q$ in $P$, we write $F(p)$ as $F_p$ and $F(p\leq q)\colon F(p)\to F(q)$ as $F_{p \subto q}$.  We call the morphisms $F_{p \subto q}$ \emph{structure maps}.   The functors $P\to \mathbf C$ form a category $\mathbf C^P$ whose morphisms are the natural transformations.

Given posets $P$ and $Q$, we regard the Cartesian product $P\times Q$ as a poset, where $(p,q)\leq (p',q')$ if and only if $p\leq p'$ and $q\leq q'$.  Let $\mathbb N^{\op}=\{1,2,\ldots\}$, regarded as a totally ordered set with the opposite of its usual order.    Let $\Simp$ denote the category of abstract simplicial complexes and simplicial maps.  By way of geometric realization, we regard $\Simp$ as a subcategory of the category $\Top$ of topological spaces and continuous maps.

A \emph{filtration} is a functor $\mathcal F\colon T\to \Simp$ for some totally ordered set $T$, such that all structure maps are inclusions.  In this paper, a \emph{bifiltration} is a functor $\mathcal F\colon \N^{\op}\times [0,\infty)\to \Simp$ such that all structure maps are inclusions.  More generally,  a \emph{semifiltration} is a functor $\mathcal F\colon \N^{\op}\times [0,\infty)\to \Simp$ such that all structure maps of the form $\mathcal F_{(p,q) \subto (p',q)}$ are inclusions.

\subsubsection{Rips Filtrations} 
\label{Sec:Rips-Filtrations}

Given a metric space $X=(X,\partial)$, we define its Rips filtration $\rips(X)\colon [0,\infty)\to \Simp$ by \[\rips(X)_r=\{\sigma\subset X\mid 0<|\sigma|<\infty,\ \diam(\sigma) \leq 2r\},\]
where \[\diam(\sigma)=\max_{x,y\in \sigma}\partial(x,y)\] is the diameter of $\sigma$.

\subsubsection{Subdivision Bifiltrations} A sequence of nested simplices \[\sigma_1 \subset \sigma_2 \subset \cdots \subset \sigma_m\] of an abstract simplicial complex $L$ is called a \emph{flag} of $L$. The set $\Bary{L}$ of all flags of $L$ forms a simplicial complex called the \emph{barycentric subdivision} of $L$.  For each $j \in \N$, define $\mathcal S(L)_j$ to be the subcomplex of $\Bary L$ spanned by all flags whose minimum element has dimension at least $j-1$.  Then in particular, $\mathcal S(L)_1 = \Bary L$.  Varying $j$ yields a filtration
\[
\mathcal S(L)\colon \mathbb N^{\mathrm{op}}\to \Simp.
\]
For any filtration ${\mathcal F:[0,\infty)\to \Simp}$, the family of filtrations $(\mathcal S(\mathcal F_r))_{r\in [0,\infty)}$ assembles into a bifiltration
\[
\subdiv \colon \mathbb N^{\mathrm{op}}\times [0,\infty) \to \Simp,
\]
which we call the \emph{subdivision filtration of $\mathcal F$} \cite{sheehyMulticoverNerve2012a}.  For $X$ a metric space, we call $\srips(X)$ the \emph{subdivision-Rips bifiltration of $X$}.

\subsubsection{Size of Bipersistent Functors}\label{Sec:Size}
We now briefly review the definition of size of certain $\Simp$-valued functors from \cite{lesnickNerveModels2024}; see \cite[Section 2.3]{lesnickNerveModels2024} for additional background, motivation, and context.

Let $\mathbb K$ be a field and let $\Vec$ be the category of vector spaces over $\mathbb K$.  Given a poset $P$ and functor $\mathcal F\colon P\to \Simp$, the usual definition of a simplicial chain complex with coefficients in $\mathbb K$ extends pointwise to yield a chain complex 
\[\cdots \to \Ch{2}{\mathcal F}\to\Ch{1}{\mathcal F}\to \Ch{0}{\mathcal F}\]
of functors $\Ch{j}{\mathcal F}\colon P\to \Vect$.  We say that $\mathcal F$ \emph{finitely presented} if each $\Ch{j}\mathcal F$ is finitely presented.  In fact, whether $\mathcal F$ is finitely presented is independent of the choice of $\mathbb K$.  For example, if $X$ is finite, then both $\rips(X)$ and $\srips(X)$ are finitely presented.

We define the \emph{size} of a finitely presented functor $\mathcal F\colon \mathbb N^{\mathrm{op}}\times [0,\infty)\to \Simp$ to be 
\begin{equation}\label{Eq:Size}
    \beta_1(\Ch{0}{\mathcal F})+\sum_{j=0}^\infty \beta_0(\Ch{j}{\mathcal F}),
\end{equation}
where $\beta_0(\Ch{j}{\mathcal F})$ and $\beta_1(\Ch{j}{\mathcal F})$ are the number of generators and relations, respectively, in a minimal presentation of $\Ch{j}{\mathcal F}$.  In \cite[Corollary 3.6]{lesnickNerveModels2024}, we showed that if $\mathcal F$ is a finitely presented semifiltration, then $\beta_1(\Ch{j}{\mathcal F})\leq \beta_0(\Ch{j}{\mathcal F})$ for all $j$.  Thus, the term $\beta_1(\Ch{0}{\mathcal F})$ in \cref{Eq:Size} can be ignored in an asymptotic analysis of the size of a semifiltration. 
\subsubsection{Weak Equivalence} 
\label{Sec:Weak-Equivalence}

For $P$ a poset and functors $\mathcal F, \mathcal F': \cat{P}\to \Top$, a natural transformation $\eta:\mathcal F\to \mathcal F'$ is called an \emph{objectwise homotopy equivalence} if for each $p\in P$, the component map $\eta_p:\mathcal F_p\to \mathcal F'_p$ is a homotopy equivalence.  If such an $\eta$ exists, we write \[\mathcal F\xlongrightarrow{\simeq}\mathcal F'.\]  We say that $\mathcal F$ and $\mathcal F'$ are \emph{weakly equivalent}, and write $\mathcal F\simeq \mathcal F'$, if they are connected by a zigzag of objectwise homotopy equivalences, as follows:
\[
\begin{tikzcd}[ampersand replacement=\&,column sep=2ex,row sep=2ex]
   \& \mathcal W_1\ar["\simeq",swap]{dl}\ar["\simeq"]{dr}  \&           \& \cdots\ar["\simeq",swap]{dl}\ar["\simeq"]{dr}  \&                \&   \mathcal W_n \ar["\simeq",swap]{dl}\ar["\simeq"]{dr}  \\
\mathcal F \&                                                             \&  \mathcal W_2 \&                                                                                                                       \& \mathcal W_{n-1} \&                                                               \&\mathcal F'.
\end{tikzcd}
\]

\subsection{Interleavings}
\label{sec:Interleavings}
A category is called \emph{thin} if for any two objects $x$ and $y$, there is at most one morphism from $x$ to $y$. For $\eps \geq 0$, let  $I^{1+\eps}$ be the thin category with object set $[0,\infty)\times \{0,1\}$ and a morphism $(r,i)\to (s,j)$ if and only if either
\begin{enumerate}
    \item $r(1+\eps) \leq s$, or
    \item $i=j$ and $r \leq s$.
\end{enumerate}

We then have functors $E^0, E^1\colon [0,\infty)\to I^{1+\eps}$ mapping $r \in [0,\infty)$ to $(r,0)$ and $(r,1)$, respectively.  For any category $\mathbf C$ and functors $\mathcal F, \mathcal F':[0,\infty)\to \mathbf C$, a \emph{(multiplicative) $(1+\eps)$-interleaving} between $\mathcal F$ and $\mathcal F'$ is a functor 
    \[
    \mathcal{Z}: I^{1+\eps} \to \mathbf C
    \]
    such that 
    $  \mathcal{Z}\circ E^0 = \mathcal F$ and $\mathcal{Z}\circ E^1= \mathcal F'$.
If such a $\mathcal Z$ exists, we say $\mathcal F$ and $\mathcal F'$ are \emph{$(1+\epsilon)$-interleaved}.  

We now extend this definition to the $2$-parameter setting, as in \cite{buchetSparseHigher2023b,lesnickNerveModels2024}: For $\eps \geq 0$, let $I^{(1,1+\eps)}$ be the thin category with object set $\Poset\times \{0, 1\}$ and a morphism $(k,r,i)\to (l,s,j)$ if and only if either
\begin{enumerate}
    \item $(k,r(1+\eps)) \leq (l,s)$, or 
    \item $i=j$ and $(k,r) \leq (l,s)$.
\end{enumerate}
As above, we have functors $E^0, E^1: \Poset \to I^{(1,1+\eps)}$ sending $(k,r) \in \Poset$ to $(k,r,0)$ and $(k,r,1)$, respectively.

    For functors $\mathcal F, \mathcal F'\colon \Poset \to \mathbf C$, we define a $(1+\eps)$-interleaving between $\mathcal F$ and $\mathcal F'$ to be a functor 
    \[
      \mathcal{Z}: I^{(1,1+\eps)} \to \mathbf C
    \]
    such that $\mathcal{Z}\circ E^0 = \mathcal F$ and $  \mathcal{Z}\circ E^1 = \mathcal F'$.  
    
\begin{remark}
    As noted in \cite[Remark 2.8]{lesnickNerveModels2024}, this definition of $2$-parameter interleaving differs slightly from the standard definition introduced in \cite{lesnickTheoryInterleaving2015} and used, e.g., in \cite{blumbergStability2Parameter2022,botnan2024bottleneck,bakke2021stability,landi2018rank}: The definition of \cite{lesnickTheoryInterleaving2015} allows for shifts in the first coordinate, and the shifts in each coordinate are additive rather than multiplicative.  
\end{remark}

Following \cite{blumbergUniversalityHomotopy2023}, for functors $\mathcal F, \mathcal G: \Poset \to \Top$, we say that $\mathcal F$ and $\mathcal G$ are \emph{$(1+\eps)$-homotopy interleaved} if there exist $(1+\eps)$-interleaved  functors $\mathcal F^\prime, \mathcal G^\prime$ such that $\mathcal F \simeq \mathcal F^\prime$ and $\mathcal G\simeq \mathcal G^\prime$.  

\begin{definition}\label{Def:Approximation}
If $\mathcal F$ and $\mathcal G$ are $(1+\eps)$-homotopy interleaved, we say that  $\mathcal G$ is a \emph{$(1+\eps)$-approximation} to $\mathcal F$.
\end{definition}

\begin{lemma}[\cite{lesnickNerveModels2024}, Proposition 2.10]\label{Lem:Inerleaving_and_Subdivision}\sloppypar
For any $\epsilon \geq 0$, if two filtrations ${\mathcal F,\mathcal F': [0,\infty) \to \Simp}$ are $(1+\epsilon)$-interleaved, then so are $\mathcal{SF}$ and $\mathcal{SF'}$.
\end{lemma}

\subsection{Nerve Models of Subdivision Filtrations}\label{Sec:Nerve_Models}
To prove \Cref{Theo:Main-Theorem}, we will use the main theorem from \cite{lesnickNerveModels2024}.  We now recall the statement in the case of $[0,\infty)$-indexed filtrations.

\begin{notation}\label{Notation:m_k}
For $\mathcal F\colon [0,\infty)\to \Simp$ a filtration and $k\geq 0$, let $m_k=m_k(\mathcal F)$ denote the number of sets $M$ of simplices in  $\colim \mathcal F =\bigcup_{t\in [0,\infty)} \mathcal F_t$ such that 
\begin{itemize}
\item $|M|\leq (k+1)$ and
\item  for some $t\in [0,\infty)$, each $\sigma\in M$ is a maximal simplex in $\mathcal F_t$.
\end{itemize}
\end{notation}

\begin{theorem}[\cite{lesnickNerveModels2024}, Theorem 1.4]\mbox{}\label{Theo:Nerve-Model}
Given a finitely presented filtration $\mathcal F\colon [0,\infty)\to \Simp$, 
\begin{itemize}
\item[(i)] there exists a semifiltration $\mathcal{NF}\colon \N^{\mathrm{op}}\times [0,\infty)\to \Simp$ weakly equivalent to $\mathcal{SF}$ whose $k$-skeleton has size $O(m_k)$ for each $k\geq 0$,
\item[(ii)] any finitely presented functor $\mathcal G\colon \N^{\mathrm{op}}\times [0,\infty)\to \Simp$ weakly equivalent to $\mathcal{SF}$ has 0-skeleton of size at least $m_0$.
\end{itemize}
\end{theorem}
Our proof of \cref{Theo:Main-Theorem} will use only \cref{Theo:Nerve-Model}\,(i).  To construct $\mathcal{NF}$, we cover each $\mathcal F_t$ by its maximal closed simplices.  This induces a functorial cover of $\mathcal {SF}$. The functor $\mathcal{NF}$ is then defined as the \emph{persistent nerve} of this functorial cover, in the sense of \cite{bauerUnifiedView2023}; see \cite[Sections 2.6 and 3.1]{lesnickNerveModels2024}. 

In fact, we can give a more explicit description of $\mathcal{NF}$ (up to canonical isomorphism): $\mathcal{NF}_{(j,r)}$ is the abstract simplicial complex consisting of sets of maximal simplices in $\mathcal F_r$ whose intersection is a simplex of dimension at least $j-1$. To define the structure maps of $\mathcal{NF}_{(j,r)}$, we choose a well-ordering on the $0$-simplices of $\colim \mathcal F$.  This induces a lexicographic well-ordering on all simplices of $\colim \mathcal F$. Then for each $r\leq s$, we define $\mathcal{NF}_{(j,r)\subto (j,s)}(\sigma)$ to be the lexicographic minimum of the set of maximal simplices in $ \mathcal F_{s}$ containing $\sigma$.  Since $\mathcal{NF}$ is a semifiltration \cite[Lemma $3.4$]{lesnickNerveModels2024}, this data fully specifies the functor. 

\subsection{Doubling Dimension and Packing}\label{Sec:Doubling}
Doubling dimension is a notion of dimension for arbitrary metric spaces that plays an important role in parts of computational geometry and topology. In computational geometry, runtime bounds for algorithms on metric data often assume bounded doubling dimension \cite{har-peledFastConstruction2005, coleSearchingDynamic2006, chanReducingCurse2018, chanSmallHopdiameter2009, abrahamRoutingNetworks2006a}.  In applied topology, size bounds for sparse filtrations on metric spaces typically also require this assumption \cite{buchetEfficientRobust2016, sheehyLinearSizeApproximations2013a,cavanna2015geometric}.  See \cite{Hei01} for an analytic introduction to doubling dimension, and \cite{clarksonNearestNeighborSearching2006} for more on the role of doubling dimension and other notions of metric dimension in computer science.

Let $X=(X, \partial)$ be a metric space.  Recall that for $x\in X$ and $r\in [0,\infty)$, the \emph{closed ball} in $X$ of radius $r$ centered at $x$ is the set \[B(x)_r= \{y\in X\mid \partial(y,x)\leq r\}.\]
All balls we consider will be closed.
\begin{definition}[{\cite{GKL03,Ass83}}]
    The \emph{doubling dimension} of $X$ is the minimum value $\lambda$ such that every ball in $X$ can be covered by at most $2^\lambda$ balls of half the radius. 
\end{definition}

A straightforward volume argument shows that the doubling dimension of $\mathbb R^m$ is $\Theta(m)$; thus, doubling dimension generalizes the ordinary Euclidean dimension, in an asymptotic sense.

Metric spaces of finite doubling dimension satisfy a \emph{packing property}, which bounds the number of well-separated points contained in a ball:
\begin{lemma}[Packing Lemma \cite{eppsteinOptimalSpanners2022a, smidWeakGap2009a}]\label{Lem:Packing-Lemma}
    Suppose $X$ has doubling dimension $d$.  If $W\subset X$ is contained in a ball of radius $r_1$ and has minimum interpoint distance at least $r_2 > 0$, then 
    \[
    |W| \leq \left( \frac{4r_1}{r_2}\right)^d.
    \]
\end{lemma}

\subsection{Well-Separated Pair Decompositions}\label{Sec:WSPDs}
\emph{Well-separated pair decompositions (WSPDs)} are tools from computational geometry for approximately representing finite metric spaces in a space-efficient way.  WSPDs were first introduced for Euclidean point sets in \cite{callahan1995decomposition}, and later generalized to finite metric spaces of bounded doubling dimension in \cite{talwar2004bypassing}.

Let  $X=(X,\partial)$ be a finite metric space.  For non-empty subsets $U,V\subset X$, let
\begin{align*}
U\otimes V&=\{\{u,v\}\mid u\in U, v\in V\},\\
\partial(U,V)&=\min_{u\in U,\,v\in V}\partial (u,v).
\end{align*}

\begin{definition}[\cite{callahan1995decomposition,talwar2004bypassing}]
For $s >0$, an \emph{$s$-well separated pair decomposition ($s$-WSPD)} of $X$  is a set of pairs of non-empty sets  $\{\{U_1,V_1\},\dots,\{U_z,V_z \}\}$ such that 
\begin{enumerate}
\item $U_i,V_i\subset X$ for all $i$,
\item $U_i\cap V_i=\emptyset$ for all $i$,
\item $(U_i\otimes V_i) \cap (U_j\otimes V_j)=\emptyset$ for all $i\ne j$,
\item $\bigcup_{i=1}^z U_i\otimes V_i= X\otimes X$,
\item $\partial(U_i,V_i)\geq s \cdot \max(\diam U_i,\diam V_i)$ for all $i$.
\end{enumerate}
\end{definition}

\begin{theorem}[{\cite[Lemma $5.1$]{har-peledFastConstruction2005}}]\label{Thm:WSPD}
If $X$ has doubling dimension $d$, then for each $s\in [1,\infty)$, then there exists an $s$-WSPD of $X$ with $O(|X|s^d)$ pairs.
\end{theorem}

\cref{Thm:WSPD} extends a result of \cite{callahan1995decomposition} from Euclidean point sets to doubling metrics, and also strengthens a result of \cite{talwar2004bypassing} by eliminating a factor in the bound depending on the spread of the metric space. 

It follows from \cref{Thm:WSPD} that finite metrics of bounded doubling dimension admit approximations with only linearly many distinct distances:
\begin{corollary}\label{Cor:WSPD_Linearly_Many_Distinct_Distances}
For fixed $\epsilon>0$ and  $X$ of bounded doubling dimension, there exists $\partial'\colon X\times X\to [0,\infty)$ such that
\begin{enumerate}
\item 
 $\im \partial'\subset \im \partial$, 
 \item 
  $|\mathop{\mathrm{im}} \partial'|=O(|X|)$, and
  \item  $\partial'(x,y)\leq \partial(x,y)\leq (1+\epsilon)\partial'(x,y)$ for all $x,y\in X$.
  \end{enumerate}
\end{corollary}
\begin{proof}
Note that if the result holds for $\epsilon$, then it holds for each $\epsilon'>\epsilon$.  Therefore, we may assume without loss of generality that $\epsilon\leq 2$.   Let \[\{\{U_1,V_1\},\ldots \{U_z,V_z\}\}\] be a $(2/\epsilon)$-WSPD of $X$ of size $O(|X|)$, which exists by \cref{Thm:WSPD}.  For $x\in X$, let $\partial'(x,x)=0$.  For $x\ne y$ in $X$, let $U_i,V_i$ be the unique pair in the WSPD with $\{x,y\} \in U_i\otimes V_i$, and let $\partial'(x,y)=\partial(U_i,V_i)$. 

 It is clear that  $\im \partial'\subset \im \partial$.  We have $|\im \partial'|=O(|X|)$ because the WSPD has $O(|X|)$ pairs.  Clearly, $\partial'(x,y)\leq \partial(x,y)$.  To see that $\partial(x,y)\leq (1+\epsilon)\partial'(x,y)$, write \[D=\max(\diam U_i,\diam V_i).\]
Choose $u\in U_i$ and $v\in V_i$ such that $\partial(u,v)=\partial(U_i,V_i)$, and assume without loss of generality that $x\in U_i$ and $y\in V_i$.  We have 
\begin{align*}
\partial(x,y)&\leq \partial(x,u)+\partial(u,v)+\partial(y,v)\\
&\leq 2D+\partial(U_i,V_i)\\
&\leq  (1+\epsilon)\partial(U_i,V_i)\\
&= (1+\epsilon)\partial'(x,y),
\end{align*}
where the third inequality follows from property (5) of a WSPD.
\end{proof}

\section{Proof of Theorem \ref{Theo:Main-Theorem}}\label{Sec:Proof_Main}
We henceforth assume that $\epsilon\in (0,\infty)$ is fixed and that $X=(X,\partial)$ is a finite metric space of doubling dimension at most a constant $d$.  As mentioned in the introduction, our strategy for proving \cref{Theo:Main-Theorem} is to approximate $\rips(X)$ by a filtration $\mathcal A(X)$ with few maximal simplices at each index, and then apply \cref{Theo:Nerve-Model}\,(i).

\subsection{Approximate Containment by Small Sets of Simplices}\label{Sec:Small_Containing_Sets_of_Cliques}
Our construction of  $\mathcal A(X)$ will hinge on the following technical lemma.

\begin{lemma}\label{Lem:Persistent-Cliques}
    For any $x\in X$ and $r \in (0, \infty)$, there exists a set $S(x,r)$ of simplices in $\mathcal{R}(X)_{r(1+\epsilon)}$, each containing $x$, such that
    \begin{enumerate}
        \item $|S(x,r)| = O(1)$,
        \item  for any simplex $\sigma \in \mathcal{R}(X)_r$ with $x\in \sigma$, we have $\sigma \subset \tau$ for some $\tau \in S(x,r)$,
    \item for $r'\geq r(1+\epsilon)$ and $\sigma\in S(x,r)$, we have  $\sigma \subset \tau$ for some $\tau \in S(x,r')$.
    \end{enumerate}
\end{lemma}
Our  proof of \Cref{Lem:Persistent-Cliques} will use the following definition:

\begin{definition}\label{Def:Packing}
    For $\alpha \in [0, \infty)$, an \emph{$\alpha$-packing} of a finite metric space $Y=(Y,\partial_{\,Y})$  is a subset $W\subset Y$ such that
\begin{enumerate}
    \item for every $y\in Y$, there exists $w \in W$ with $\partial_{\,Y}(y,w) \leq \alpha$
    \item for every pair of distinct elements $w, w' \in W$, we have $\partial_{\,Y}(w,w') \geq \alpha$.
\end{enumerate}
\end{definition}

\begin{remark}\label{Rem:Greedy_Packing_Alg}
Given a point $y\in Y$, a simple greedy algorithm computes an $\alpha$-packing $W\subset Y$ containing $y$ in time $O(|Y||W|)$: Initialize $W=\{y\}$, and iterate through the remaining elements of $Y$ in arbitrary order, adding each element $z\in Y$ to $W$ if and only if $\partial_{\,Y}(z,W) > \alpha$. 
\end{remark}

\begin{proof}[Proof of \Cref{Lem:Persistent-Cliques}]
    Let $W$ be an $(r\epsilon/2)$-packing of $B(x)_{2r}$ containing $x$.  Let $\Gamma$ denote the set of simplices of $\mathcal R(W)_{r(1+\epsilon/2)}$ that are maximal and contain $x$.  For each $\sigma\in \Gamma$, let \[\overline  \sigma=\{y\in B(x)_{2r} \mid \partial(y,\sigma)\leq r\epsilon/2\}.\]
   Observe that $\overline  \sigma\in \mathcal R(X)_{r(1+\epsilon)}$: If $y,z\in \overline  \sigma$, then $\partial(y,v)\leq r\epsilon/2$ and $\partial(z,w)\leq r\epsilon/2$ for some $v,w\in \sigma$, so by the triangle inequality, \[\partial(y,z)\leq \partial(y,v)+\partial(v,w)+\partial(w,z)\leq 2r(1+\epsilon/2)+r\epsilon=2r(1+\epsilon).\]
   Let  \[S(x,r)=\{\overline  \sigma\mid \sigma\in \Gamma\}.\] 
  
  Note that if $|W|=1$, then $|S(x,r)|=1$.  Therefore, to check that $|S(x,r)|=O(1)$, it suffices to consider the case $|W|>1$.  Then, since distinct points in $B(x)_{2r}$ have distance at most $4r$, we have $r\epsilon/2\leq 4r$, which implies that  $\epsilon\leq 8<16$.  Hence, letting $d'\leq d$ denote the doubling dimension of $X$, \Cref{Lem:Packing-Lemma} gives that 
  \[|W|\leq (16/\epsilon)^{d'}\leq (16/\epsilon)^d=O(1).\]
Therefore, 
 \[|S(x,r)|\leq |\Gamma|\leq 2^{|W|}\leq 2^{(16/\epsilon)^{d}}=O(1).\]  
Thus, $S(x,r)$ satisfies (1).  

To see that $S(x,r)$ satisfies (2), consider $\sigma \in \mathcal{R}(X)_r$ with $x\in \sigma$, and let
    \[
    \omega = \{ w\in W \mid \partial(w,\sigma) \leq r\epsilon/2\}.
    \]
    Then by the triangle inequality, $x\in \omega\in \mathcal R(W)_{r(1+\epsilon/2)}$.  Thus, $\omega$ is contained in some $\nu\in \Gamma$.  Since $\sigma\subset B(x)_{2r}$ and $W$ is an $(r\epsilon/2)$-packing  of $B(x)_{2r}$, we have $\sigma \subset \overline  \nu\in S(x,r)$.  This establishes (2).  

To verify (3), note that for each $\sigma\in S(x,r)$, we have \[x\in \sigma\in \mathcal{R}(X)_{r(1+\epsilon)}\subset  \mathcal{R}(X)_{r'}.\]  Therefore, $\sigma\subset \tau$ for some $\tau\in S(x,r')$.  
\end{proof}

\subsection{Construction of the Approximating Filtration \texorpdfstring{$\mathcal{A}(X)$}{A(X)}}

\subsubsection{A Simplified Construction}
Since our construction of $\mathcal A(X)$ is somewhat technical, to aid understanding, we first give a simplified version that suffices to prove a weaker form of \cref{Theo:Main-Theorem}.  

Recall that the \emph{birth index} $b_\sigma$ of a simplex $\sigma \in \colim \rips(X)$ is defined as \[b_{\sigma}=\min\,\{ r \in [0,\infty) \mid \sigma\in \rips(X)_r\}.\] 
We can use \cref{Lem:Persistent-Cliques} in a straightforward way to define an approximation $\mathcal A'(X)$ of $\mathcal R(X)$, as follows: Let $E$ be the set of all edges in $\colim \rips(X)$.  For each $e\in E$, choose a vertex $x_e$ incident to $e$,  and a set of simplices $S(x_e,b_e)$ as in \cref{Lem:Persistent-Cliques}.  Given a set of simplices $Z$ in $\colim \rips(X)$, let $\overline{Z}$ denote the simplicial complex consisting of the simplices of $Z$ and all of their faces.
We define the filtration $\mathcal A'(X)\colon [0,\infty)\to \Simp$ by \[\mathcal A'(X)_r=X\cup \left(\bigcup_{\left\{e\in E\,\mid\, b_e\leq r\right\}}  \overline{\mathcal S(x_e,b_e)}\right).\]

\cref{Lem:Persistent-Cliques}\,(1) implies that $\mathcal A'(X)$ has $O(|X|^{2})$ distinct maximal simplices across all indices, which yields the naive bound $m_k(A'(X))=O(|X|^{2(k+1)})$ for all $k\geq 0$ (see \cref{Notation:m_k}).  Given this, a simplification of our proof of \cref{Theo:Main-Theorem} below gives that $\mathcal {NA'}(X)$ is a $(1+\eps)$-approximation to $\srips(X)$ whose $k$-skeleton has size $O(|X|^{2(k+1)})$.  

To achieve the stronger size bound of \cref{Theo:Main-Theorem}, we need finer control over the number of maximal simplices at individual indices.  To this end, in our definition of $\mathcal A(X)$, we will take the values of $r$ at which we select the sets $S(x_e,r)$ to be separated by a factor of $1+\epsilon$, so that we can apply \cref{Lem:Persistent-Cliques}\,(3).  

\subsubsection{Definition of $\mathcal A(X)$}\label{Sec:Def_of_A(X)}
If $|X|\leq 1$, we define $\mathcal A(X)=\mathcal R(X)$.  Now assume $|X|>1$.  Let \[R=\{b_e\mid e\in E\}.\]
Define \[Q=\{q_1,\ldots,q_n\}\subset (0,\infty)\] inductively as follows: Let $q_1=\min R$.  Assuming $q_i$ has been defined, let 
\[
R^i=\{r\in R \mid r> q_{i}\}.
\]  
If $R^i\ne \emptyset$, then define 
\[
q_{i+1}=\max\,\{\min R^i,q_i(1+\epsilon)\}.
\]  
If $R^i= \emptyset$, then $q_i=q_n$ is the largest element of $Q$.   

\begin{lemma}\mbox{}\label{Lem:q_i_properties}
\begin{itemize}
\item[(i)] For each $i<n$, we have $q_i(1+\epsilon)\leq q_{i+1}$.
\item[(ii)] For each $r\in R$, there exists $q_i\in Q$ such that $r\leq q_i \leq r(1+\epsilon)$.
\end{itemize}  
\end{lemma}

\begin{proof}
Statement (i) follows directly from the definition of $Q$.   To prove (ii), let \[q_i=\min\, \{q\in Q\mid r\leq q\}.\] 
Note that $q_i$ is well-defined, since by construction it is the minimum of a non-empty set. If $r =q_i$ then the result clearly holds, so assume $r<q_i$, in which case $i>1$ and $ q_{i-1}<r$.  If $ r(1+\epsilon)<q_i$, then $q_{i-1}(1+\epsilon)<q_{i}$, which implies that $q_i=\min R^{i-1}$. This is a contradiction since $r\in R^{i-1}$ and $r<q_i$.  Thus, $r\leq q_i \leq r(1+\epsilon)$.  
\end{proof}

\begin{proposition}\label{Prop:|Q|=O(|X|)}
We have $|Q|=O(|X|)$.
\end{proposition}

\begin{proof}
For $\partial'\colon X\times X\to [0,\infty)$ as in \cref{Cor:WSPD_Linearly_Many_Distinct_Distances}, let $R'=\im \partial'\setminus \{0\}$.  Note that
\begin{enumerate}
\item  $R'\subset R$, 
\item $|R'|=O(|X|)$, and
\item  for each $r\in R$, there exists $r'\in R'$ with $r'\leq r\leq r'(1+\epsilon)$. 
\end{enumerate}
 By the definition of $Q$, for each $q\in Q$ there exists $r\in R$ with $r\leq q\leq r(1+\epsilon)$.  Thus, there exists $r'\in R'$ with 
\begin{equation}\label{q_and_r'}
r'\leq q\leq r'(1+\epsilon)^2.
\end{equation}
  Moreover, by \cref{Lem:q_i_properties}\,(i), for each $r'\in R$, there can be at most three elements of $Q$ which satisfy \cref{q_and_r'}.  Thus $|Q|\leq 3|R'|=O(|X|)$.  
\end{proof}

Let $q_0=0$.  For $i\in \{1,\ldots,n\}$, let \[E_i=\left\{e\in E\mid b_e\in (q_{i-1},q_i]\right\}.\]  For each $e\in E_i$, choose a vertex $x_e$ incident to $e$, and a set $S(x_e,q_i)$ as in \cref{Lem:Persistent-Cliques}.  We will say that $x_e$ is \emph{chosen at }$q_i$.  Let \[S(q_i)=\bigcup_{e\in E_i} S(x_e,q_i).\]   
We define the filtration $\mathcal A(X)\colon [0,\infty)\to \Simp$ by \[\mathcal A(X)_r=X\cup \left(\bigcup_{q_i\leq r} \overline{S(q_i)}\right).\]

\subsection{Properties of \texorpdfstring{$\mathcal A(X)$}{A(X)}}

\begin{proposition}\label{prop:interleaved}
The filtrations $\mathcal A(X)$ and $\rips(X)$ are $(1+\epsilon)$-interleaved.  
\end{proposition}

\begin{proof}
To simplify notation, we write $\mathcal A\coloneqq \mathcal A(X)$ and $\mathcal R\coloneqq \mathcal R(X)$.  We will prove the result by showing that $\mathcal A_r\subset \rips_{r(1+\epsilon)}$ and $\mathcal \rips_r\subset \mathcal A_{r(1+\epsilon)}$ for each $r\in [0,\infty)$.  Since the 0-skeletons of $\rips_{r}$ and $\mathcal A_{r(1+\epsilon)}$ are both $X$, to show $\mathcal A_r\subset \rips_{r(1+\epsilon)}$, it suffices to show that if $\sigma\in \mathcal A_{r}$ and $\dim(\sigma)\geq 1$, then $\sigma\in \rips_{r(1+\epsilon)}$.  But for some $q\leq r$, we have  \[\sigma\in S(x,q)\subset  \rips_{q(1+\epsilon)}\subset \rips_{r(1+\epsilon)},\] so indeed $\sigma\in  \rips_{r(1+\epsilon)}$.

To show that  $\mathcal \rips_r\subset \mathcal A_{r(1+\epsilon)}$ for each $r\in [0,\infty)$, it is enough to show this for each $r\in R$.  Assuming $r\in R$,  \cref{Lem:q_i_properties}\,(ii) implies that $r\leq q \leq r(1+\epsilon)$ for some $q\in Q$.  We claim that $\rips_{q}\subset \mathcal A_{q}$.  To prove the claim, it suffices to show that if $\sigma\in \rips_{q}$ and $\dim(\sigma)\geq 1$, then $\sigma\in \mathcal A_{q}$.  Let $e$ be an edge in $\sigma$ such that the value $b_e$ is maximized.  Then $e\in E_i$ for some $i$ with $q_i\leq q$.  Note that $x_e \in \sigma\in \rips_{q_i}$. \cref{Lem:Persistent-Cliques}\,(2) then implies that $\sigma\subset \tau$ for some $\tau\in S(x_e,q_i)$.  Since $S(x_e,q_i)\subset S(q_i)\subset \mathcal A_{q_i}\subset \mathcal A_{q}$, we have $\tau\in \mathcal A_{q}$.  Therefore, $\sigma\in \mathcal A_{q}$, as desired.   We thus have 
\[\rips_r\subset \rips_q\subset \mathcal A_{q}\subset \mathcal A_{r(1+\epsilon)}.\qedhere\]
\end{proof}

\begin{lemma}\label{Lem:Num_Distinct_Complexes}
The filtration $\mathcal A(X)$ has $O(|X|)$ distinct simplicial complexes.  
\end{lemma}

\begin{proof}
It is clear from the definition of $\mathcal A(X)$ that for each $r\in [0,\infty)$, $\mathcal A(X)_r=\mathcal A(X)_q$ for some $q\in Q\cup \{0\}$.  Since $|Q|=O(|X|)$ by \cref{Prop:|Q|=O(|X|)}, the result follows.
\end{proof}

\begin{lemma}\label{Lem:Num_Max_at_an_Index}
For each $r \in [0, \infty)$, the complex $\mathcal A(X)_r$ has $O(|X|)$ maximal simplices, where the asymptotic notation hides a constant independent of $r$.
\end{lemma}

\begin{proof}
Let \[Y=\{x\in X\mid S(x,q)\subset \mathcal S(q) \textup{ for some } q\in Q\textup{ with }q\leq r\}.\]
For $x\in Y$, let \[q_x=\max\,\{q\leq r\mid S(x,q)\subset \mathcal S(q)\}.\]
By \cref{Lem:q_i_properties}\,(i) and \cref{Lem:Persistent-Cliques}\,(3), we have \[\mathcal A(X)_r=X\cup \left(\bigcup_{x\in Y} \overline{S(x,q_x)}\right).\]
Thus, any maximal simplex of $\mathcal A(X)_r$ is contained in the set of simplices \[X\cup \left(\bigcup_{x\in Y} S(x,q_x)\right),\] which has size $O(|X|)$, since by \cref{Lem:Persistent-Cliques}\,(1) we have $|S(x,q_x)|=O(1)$ for each $x\in Y$.
\end{proof}

\begin{proposition}\label{prop:max_simplex_bound}
For each $k\geq 0$, we have $m_k(\mathcal A(X))=O(|X|^{k+2})$. 
\end{proposition}

\begin{proof}
We call a set with $j$ elements a \emph{$j$-set}.  For each $r\in [0,\infty)$,  \cref{Lem:Num_Max_at_an_Index} tells us that $\mathcal A(X)_{r}$ has $O(|X|)$ maximal simplices.  Thus, for each $j \geq 0$, the number of distinct $(j+1)$-sets of maximal simplices in $\mathcal A(X)_r$ is $O(|X|^{j+1})$.  Hence, \cref{Lem:Num_Distinct_Complexes} implies that as $r$ varies, we encounter a total of $O(|X|^{j+2})$ distinct $(j+1)$-sets of maximal simplices among all simplicial complexes in $\mathcal A(X)$.    

By \cref{Lem:Num_Max_at_an_Index}, we may choose a constant $c\in \N$ such that the number of maximal simplices in $\mathcal A(X)_r$ is at most $c|X|$ for each $r$. We then have \[m_k(\mathcal A(X))=O\left(\sum_{j=0}^{\min(k,c|X|)} |X|^{j+2}\right)=O(|X|^{k+2}).\qedhere\]
\end{proof}

\begin{proof}[Proof of \Cref{Theo:Main-Theorem}]
By \cref{prop:interleaved}, $\mathcal A(X)$ and $\rips(X)$ are $(1+\epsilon)$-interleaved, so \cref{Lem:Inerleaving_and_Subdivision} implies that $\mathcal {SA}(X)$ and $\srips(X)$ are $(1+\eps)$-interleaved.  In view of the bound on $m_k(\mathcal A(X))$ given by \cref{prop:max_simplex_bound}, the theorem follows by applying  \Cref{Theo:Nerve-Model}\,(i) to $\mathcal A(X)$.  
\end{proof}

\section{Computation}
\label{Sec:Computation}\sloppypar
In this section, we give a straightforward algorithm to compute the $k$-skeleton of $\mathcal{N A}(X)$.  For $k\geq 1$ constant, the algorithm runs in time $O(|X|^{k+3})$. 

\subsection{Problem Specification}
Let $\mathcal N$ denote the $k$-skeleton of $\mathcal{N A}(X)$ and let  $P=\mathbb N^{\op}\times [0,\infty)$.  Concretely, our aim will be to compute sets $G=\sqcup_{j=0}^k G^j$ and $H$, defined below, that determine $\mathcal{N}$ up to isomorphism.  We call the pair $(G,H)$ a \emph{presentation of} $\mathcal N$, by analogy with the usual notion of a presentation of a persistence module.  One could generalize our approach to define presentations of arbitrary $\Simp$-valued functors, but we will not do so here.

We define 
\[G^j=\setbuilder*{g\in \bigsqcup_{p\in P}\mathcal N_{p};  \dim(g)=j,\ g \not\in \mathop{\mathrm{im}} \mathcal N_{p\subto \gr{g}} \textup{ for any }p< \gr{g}},\]
where $\gr{g}$ is the unique element of $P$ such that $g\in \mathcal N_{\gr{g}}$.  
We can give a more explicit description of $G^j$, as follows:  
 Recall that $Q=\{q_1,\ldots,q_n\}$ and that $q_0=0$.  For $i\in \{0,\ldots,n\}$, let $M^j_i$ denote the set of $(j+1)$-sets of maximal simplices of $\mathcal A(X)_{q_i}$ with non-empty intersection.  Let $L^j_0=M^j_{0}$, and for $i\in \{1,\ldots,n\}$, let $L^j_{i}=M^j_{i}\setminus M^j_{i-1}$.  We then have 
\begin{equation}\label{eq:Gj}
G^j=\bigsqcup_{i=0}^n L^j_{i},
\end{equation}
where for each $g\in L^j_i$,  \[\gr{g}=(|\cap_{\sigma\in g}\sigma|,q_i).\]
  
We next define $H$.  Noting that elements of $G^0$ are simplices in $\mathcal A(X)$, we say $\sigma\in G^0$ is \emph{dominated} if there exists $\tau\in G^0$ with $\sigma\subsetneq \tau$.  If $\sigma$ is dominated, let 
\begin{equation}\label{Eq:Sigma_Prime}
\sigma'=\min\, \{\tau \in G^0\mid \sigma \subsetneq \tau\},
\end{equation}
where the minimum is taken with respect to the second coordinate of $\gr{\tau}$, with ties broken by taking the minimum with respect to the lexicographical well-ordering on simplices of $\colim \mathcal A(X)$  (see \cref{Sec:Nerve_Models}).  Let 
\[H =\{(\sigma,\sigma')\mid \sigma \in G^0 \textup{ is dominated}\}.\]
We think of $(\sigma,\sigma')\in H$ as a relation that identifies $\sigma$ and $\sigma'$ at index $\gr{\sigma}\vee \gr{\sigma'}$, where $\vee$ denotes the least upper bound in $\mathbb N^{\op}\times [0,\infty)$.   

\begin{proposition}
The pair $(G,H)$ determines $\mathcal{N}$ up to isomorphism.  
\end{proposition}

\begin{proof}[Sketch of Proof]
 For $p\in P$, let 
\begin{align*}
G_p&=\{g\in G\mid \gr{g}\leq p\},\\ 
G^0_p&=G_p\cap G^0,\\
H_p&=\{(\sigma,\sigma')\in H\mid \gr{\sigma}\vee \gr{\sigma'}\leq p\}.
\end{align*}
 Formally, $H_p$ is a relation on $G^0_p$; let $(G^0/H)_p$ denote the quotient of $G^0$ by the equivalence relation that $H_p$ generates.  For $v\in G^0_p$, let $\hat v$ denote its equivalence class in $(G^0/H)_p$.  We say that  $g=\{\sigma_1,\ldots,\sigma_l\}\in G_p$ is \emph{non-degenerate at $p$} if $\hat \sigma_i\ne \hat \sigma_j$ for $i\ne j$.
If $g$ is non-degenerate at $p$, let $\hat g=\{\hat \sigma_1,\ldots,\hat \sigma_l\}$.  It can be checked that \[(G/H)_p\coloneqq \{\hat g \mid g \in G_p \textup{ is non-degenerate}\}\] is a simplicial complex with 0-skeleton $(G^0/H)_p$, and that as $p$ varies, the complexes $(G/H)_p$ assemble into a functor $G/H\colon P\to \Simp$ such that $G/H\cong \mathcal N$.  
\end{proof}

\begin{remark}
The presentation $(G,H)$ has an algebraic interpretation: Each $G^j$ can be identified with a minimal set of generators for the chain complex $\Ch{j}{\mathcal N}$, while $H$ can be identified with minimal set of relations for $\Ch{0}{\mathcal N}$.  In particular, $|G|+|H|$ equals the size of $\mathcal N$, as defined in \cref{Sec:Size}.  In fact, our algorithm for computing $(G,H)$ extends straightforwardly to compute minimal presentations of $\Ch{j}{\mathcal{N}}$ for all $j\in \{0,\ldots,k\}$ in the same asymptotic time, but we will omit the details for brevity's sake.
\end{remark}

\subsection{Algorithm and Complexity Analysis}  
We assume that a total order on $X$ is fixed, and that $X$ is specified via a distance matrix.  To compute $(G,H)$, we begin by computing $Q$ in time $O(|X|^2 \log |X|)$.  The cost of this is dominated by the time required to sort the non-zero distances of $X$ in increasing order.  

We next consider the computation of $G^0$.  Since $L^0_0=X$, by \cref{eq:Gj} it suffices to compute the sets $L^0_i$ for each $i>0$.  Note that if $i>0$, then $L^0_{i}\subset S(q_i)$.  Our approach to computing the sets $L^0_{i}$ involves computing the sets $S(q_i)$.

For any $q\in Q$, computing $S(q)$ amounts to computing $S(x,q)$ for each $x$ is chosen at $q$.  We compute $S(x,q)$ by following the construction in the proof of \cref{Lem:Persistent-Cliques}: We first compute the $(q\epsilon/2)$-packing $W$ of the ball $B(x,2q)$.  Since $W$ has constant size, this can be done in time $O(|X|)$ using the algorithm outlined in \cref{Rem:Greedy_Packing_Alg}. Given $W$, we can construct the set $\Gamma$ using, e.g., the Bron–Kerbosch algorithm \cite{bron1973algorithm} or a brute force approach.  Since $\Gamma$ consists of a constant number of sets, each of constant size, naively computing $S(x,q)$ from $\Gamma$ requires time $O(|X|)$ in the worst case.  Thus, computing $S(x,q)$ requires time $O(|X|)$ in total.  Since $S(q)$ is the union of at most $|X|$ sets $S(x,q)$ of constant size, and $|Q|=O(|X|)$ by \cref{Prop:|Q|=O(|X|)}, computing the sets $S(q)$ for all $q\in Q$ requires total time $O(|X|^3)$.

To compute the sets $L^0_{i}$ from the sets $S(q_i)$, we proceed inductively with respect to $i$.  Assume we are given $M^0$ and $S(q_{i+1})$, with the vertices of each simplex in each set sorted in increasing order.  We then compute $L^0_{i+1}$ and $M^0_{i+1}$, as follows:  Note that $M^0_{i+1}$ is the set of maximal simplices in $M^0_{i}\cup S(q_{i+1})$.  Given a pair of simplices in $M^0_i\cup S(q_{i+1})$, we can naively check whether one simplex contains the other in time $O(|X|)$ by iterating through the two lists of vertices.  By performing such containment tests on pairs of simplices in $S(q_{i+1})$, we first remove all non-maximal simplices in $S(q_{i+1})$.  Then, by performing further containment tests on pairs $\{\sigma,\mu\}$ where $\sigma\in S(q_{i+1})$ is maximal and $\mu\in M^0_i$, we obtain $L^0_{i+1}$ as well as $M^0_{i+1}$.    

\cref{Lem:Num_Max_at_an_Index} and \cref{Lem:Persistent-Cliques}\,(1) imply that for each $i$, $|M^0_i|=O(|X|)$ and $|S(q_{i+1})|=O(|X|)$.  In addition, we have $|Q|=O(|X|)$ by \cref{Prop:|Q|=O(|X|)}.  Thus, as $i$ varies, the total cost of all containment tests between pairs of simplices in $M^0_i\cup S(q_{i+1})$ is $O(|X|^4)$.  Hence, the total cost of computing the sets $L^0_i$ for all $i$ is $O(|X|^4)=O(|X|^{k+3})$.

Computing $H$ amounts to identifying, for each $i<n$, each simplex $\sigma \in M^0_i$ that is a proper subset of some $\tau \in L^0_{i+1}$, and for each such $\sigma$, recording the lexicographically minimal such $\tau$ (which is the simplex $\sigma'$ of \cref{Eq:Sigma_Prime}).  Thus, our computation of the sets $M^0_{i+1}$ and $L^0_{i+1}$ via containment tests extends readily to compute $H$, without increasing the asymptotic cost.  

For $j\in \{1,\ldots,k\}$, computing $L^j_i$ amounts to identifying all elements (i.e., $(j+1)$-sets) of $M^j_i$ such that at least one element of the $(j+1)$-set belongs to $L^0_j$.  As we assume $k$ is constant, the intersection of a $(j+1)$-set of simplices can be computed straightforwardly in time $O(|X|)$.  Thus, as $|M^j_i|=O(|X|^{j+1})$ and $|Q|=O(|X|)$, to compute the sets $L^j_i$ for all $i$ and all $j\in \{1,\ldots,k\}$, it suffices to compute the intersection of $O(|X|^{k+2})$ different $(j+1)$-sets, which requires time $O(|X|^{k+3})$.

\section*{Acknowledgements}
We thank Sariel Har-Peled for a helpful exchange about \cref{Cor:WSPD_Linearly_Many_Distinct_Distances}, and in particular, for pointing out that this follows from the results in \cite[Section 5]{har-peledFastConstruction2005} on well-separated pair decompositions.

\bibliographystyle{plainurl}
\bibliography{bibliography}

\begin{thebibliography}{10}

\bibitem{abrahamRoutingNetworks2006a}
I.~Abraham, C.~Gavoille, A.V. Goldberg, and D.~Malkhi.
\newblock Routing in {{Networks}} with {{Low Doubling Dimension}}.
\newblock In {\em 26th {{IEEE International Conference}} on {{Distributed Computing Systems}} ({{ICDCS}} '06)}, pages 75--75, 2006.
\newblock \href {https://doi.org/10.1109/ICDCS.2006.72} {\path{doi:10.1109/ICDCS.2006.72}}.

\bibitem{Ass83}
Patrice Assouad.
\newblock Plongements lipschitziens dans $\mathbb{R}^n$.
\newblock {\em Bulletin de la Société Mathématique de France}, 79:429--448, 1983.
\newblock \href {https://doi.org/10.24033/bsmf.1997} {\path{doi:10.24033/bsmf.1997}}.

\bibitem{bakke2021stability}
H{\aa}vard Bakke~Bjerkevik.
\newblock On the stability of interval decomposable persistence modules.
\newblock {\em Discrete \& Computational Geometry}, 66(1):92--121, 2021.
\newblock \href {https://doi.org/10.1007/s00454-021-00298-0} {\path{doi:10.1007/s00454-021-00298-0}}.

\bibitem{bauerUnifiedView2023}
Ulrich Bauer, Michael Kerber, Fabian Roll, and Alexander Rolle.
\newblock A unified view on the functorial nerve theorem and its variations.
\newblock {\em Expositiones Mathematicae}, 2023.
\newblock \href {https://doi.org/10.1016/j.exmath.2023.04.005} {\path{doi:10.1016/j.exmath.2023.04.005}}.

\bibitem{blumbergUniversalityHomotopy2023}
Andrew Blumberg and Michael Lesnick.
\newblock Universality of the homotopy interleaving distance.
\newblock {\em Transactions of the American Mathematical Society}, 376(12):8269--8307, 2023.
\newblock \href {https://doi.org/10.1090/tran/8738} {\path{doi:10.1090/tran/8738}}.

\bibitem{blumbergRobustStatistics2014}
Andrew~J. Blumberg, Itamar Gal, Michael~A. Mandell, and Matthew Pancia.
\newblock Robust {{Statistics}}, {{Hypothesis Testing}}, and {{Confidence Intervals}} for {{Persistent Homology}} on {{Metric Measure Spaces}}.
\newblock {\em Foundations of Computational Mathematics}, 14(4):745--789, 2014.
\newblock \href {https://doi.org/10.1007/s10208-014-9201-4} {\path{doi:10.1007/s10208-014-9201-4}}.

\bibitem{blumbergStability2Parameter2022}
Andrew~J. Blumberg and Michael Lesnick.
\newblock Stability of 2-{Parameter} {Persistent} {Homology}.
\newblock {\em Foundations of Computational Mathematics}, 2022.
\newblock \href {https://doi.org/10.1007/s10208-022-09576-6} {\path{doi:10.1007/s10208-022-09576-6}}.

\bibitem{botnanIntroductionMultiparameter2023}
Magnus Botnan and Michael Lesnick.
\newblock An introduction to multiparameter persistence.
\newblock In {\em Representations of Algebras and Related Structures}, pages 77--150. European Mathematical Society, 2023.
\newblock \href {https://doi.org/10.4171/ECR/19} {\path{doi:10.4171/ECR/19}}.

\bibitem{botnan2024bottleneck}
Magnus~Bakke Botnan, Steffen Oppermann, Steve Oudot, and Luis Scoccola.
\newblock On the bottleneck stability of rank decompositions of multi-parameter persistence modules.
\newblock {\em Advances in Mathematics}, 451:109780, 2024.
\newblock \href {https://doi.org/10.1016/j.aim.2024.109780} {\path{doi:10.1016/j.aim.2024.109780}}.

\bibitem{botnan2015approximating}
Magnus~Bakke Botnan and Gard Spreemann.
\newblock Approximating persistent homology in {Euclidean} space through collapses.
\newblock {\em Applicable Algebra in Engineering, Communication and Computing}, 26(1-2):73--101, 2015.
\newblock \href {https://doi.org/10.1007/s00200-014-0247-y} {\path{doi:10.1007/s00200-014-0247-y}}.

\bibitem{brehm2018sparips}
Bernhard Brehm and Hanne Hardering.
\newblock Sparips, 2018.
\newblock \textit{arXiv preprint}.
\newblock \href {https://arxiv.org/abs/1807.09982} {\path{arXiv:1807.09982}}.

\bibitem{bron1973algorithm}
Coen Bron and Joep Kerbosch.
\newblock Algorithm 457: finding all cliques of an undirected graph.
\newblock {\em Communications of the ACM}, 16(9):575--577, 1973.
\newblock \href {https://doi.org/10.1145/362342.362367} {\path{doi:10.1145/362342.362367}}.

\bibitem{brun2019sparse}
Morten Brun and Nello Blaser.
\newblock Sparse {Dowker} nerves.
\newblock {\em Journal of Applied and Computational Topology}, 3(1):1--28, 2019.
\newblock \href {https://doi.org/10.1007/s41468-019-00028-9} {\path{doi:10.1007/s41468-019-00028-9}}.

\bibitem{buchetSparseHigher2023b}
Micka\"{e}l Buchet, Bianca B.~Dornelas, and Michael Kerber.
\newblock {Sparse Higher Order \v{C}ech Filtrations}.
\newblock In {\em 39th International Symposium on Computational Geometry (SoCG 2023)}, pages 20:1--20:17, 2023.
\newblock \href {https://doi.org/10.4230/LIPIcs.SoCG.2023.20} {\path{doi:10.4230/LIPIcs.SoCG.2023.20}}.

\bibitem{buchetEfficientRobust2016}
Micka{\"e}l Buchet, Fr{\'e}d{\'e}ric Chazal, Steve~Y. Oudot, and Donald~R. Sheehy.
\newblock Efficient and robust persistent homology for measures.
\newblock {\em Computational Geometry}, 58:70--96, 2016.
\newblock \href {https://doi.org/10.1016/j.comgeo.2016.07.001} {\path{doi:10.1016/j.comgeo.2016.07.001}}.

\bibitem{callahan1995decomposition}
Paul~B Callahan and S~Rao Kosaraju.
\newblock A decomposition of multidimensional point sets with applications to $k$-nearest-neighbors and $n$-body potential fields.
\newblock {\em Journal of the ACM (JACM)}, 42(1):67--90, 1995.
\newblock \href {https://doi.org/10.1145/200836.200853} {\path{doi:10.1145/200836.200853}}.

\bibitem{carlssonTheoryMultidimensional2009}
Gunnar Carlsson and Afra Zomorodian.
\newblock The {{Theory}} of {{Multidimensional Persistence}}.
\newblock {\em Discrete \& Computational Geometry}, 42(1):71--93, 2009.
\newblock \href {https://doi.org/10.1007/s00454-009-9176-0} {\path{doi:10.1007/s00454-009-9176-0}}.

\bibitem{cavanna2015geometric}
Nicholas~J. Cavanna, Mahmoodreza Jahanseir, and Donald~R. Sheehy.
\newblock A geometric perspective on sparse filtrations.
\newblock In {\em Proceedings of the Canadian Conference on Computational Geometry ({CCCG} 2015)}, pages 116--121, 2015.
\newblock URL: \url{https://cccg.ca/proceedings/2015/01.pdf}.

\bibitem{chanSmallHopdiameter2009}
T.-H.~Hubert Chan and Anupam Gupta.
\newblock Small {{Hop-diameter Sparse Spanners}} for {{Doubling Metrics}}.
\newblock {\em Discrete \& Computational Geometry}, 41(1):28--44, 2009.
\newblock \href {https://doi.org/10.1007/s00454-008-9115-5} {\path{doi:10.1007/s00454-008-9115-5}}.

\bibitem{chanReducingCurse2018}
T.-H.~Hubert Chan and Shaofeng H.-C. Jiang.
\newblock Reducing {{Curse}} of {{Dimensionality}}: {{Improved PTAS}} for {{TSP}} (with {{Neighborhoods}}) in {{Doubling Metrics}}.
\newblock {\em ACM Transactions on Algorithms (TALG)}, 14(1):9:1--9:18, 2018.
\newblock \href {https://doi.org/10.1145/3158232} {\path{doi:10.1145/3158232}}.

\bibitem{chazal2009gromov}
Fr{\'e}d{\'e}ric Chazal, David Cohen-Steiner, Leonidas~J. Guibas, Facundo M{\'e}moli, and Steve Oudot.
\newblock {Gromov-Hausdorff Stable Signatures for Shapes using Persistence}.
\newblock {\em {Computer Graphics Forum}}, 28(5):1393--1403, 2009.
\newblock \href {https://doi.org/10.1111/j.1467-8659.2009.01516.x} {\path{doi:10.1111/j.1467-8659.2009.01516.x}}.

\bibitem{chazal2014persistence}
Fr{\'e}d{\'e}ric Chazal, Vin De~Silva, and Steve Oudot.
\newblock Persistence stability for geometric complexes.
\newblock {\em Geometriae Dedicata}, 173(1):193--214, 2014.
\newblock \href {https://doi.org/10.1007/s10711-013-9937-z} {\path{doi:10.1007/s10711-013-9937-z}}.

\bibitem{choudhary2019improved}
Aruni Choudhary, Michael Kerber, and Sharath Raghvendra.
\newblock Improved topological approximations by digitization.
\newblock In {\em Proceedings of the Thirtieth Annual ACM-SIAM Symposium on Discrete Algorithms}, pages 2675--2688, 2019.
\newblock \href {https://doi.org/10.1137/1.9781611975482.166} {\path{doi:10.1137/1.9781611975482.166}}.

\bibitem{choudhary2019polynomial}
Aruni Choudhary, Michael Kerber, and Sharath Raghvendra.
\newblock Polynomial-sized topological approximations using the permutahedron.
\newblock {\em Discrete \& Computational Geometry}, 61(1):42--80, 2019.
\newblock \href {https://doi.org/10.1007/s00454-017-9951-2} {\path{doi:10.1007/s00454-017-9951-2}}.

\bibitem{choudhary2021improved}
Aruni Choudhary, Michael Kerber, and Sharath Raghvendra.
\newblock Improved approximate {Rips} filtrations with shifted integer lattices and cubical complexes.
\newblock {\em Journal of Applied and Computational Topology}, pages 1--34, 2021.
\newblock \href {https://doi.org/10.1007/s41468-021-00072-4} {\path{doi:10.1007/s41468-021-00072-4}}.

\bibitem{clarksonNearestNeighborSearching2006}
Kenneth~L. Clarkson.
\newblock Nearest-{{Neighbor Searching}} and {{Metric Space Dimensions}}.
\newblock In Gregory Shakhnarovich, Trevor Darrell, and Piotr Indyk, editors, {\em Nearest-{{Neighbor Methods}} in {{Learning}} and {{Vision}}}, pages 15--60. The MIT Press, 2006.
\newblock \href {https://doi.org/10.7551/mitpress/4908.003.0005} {\path{doi:10.7551/mitpress/4908.003.0005}}.

\bibitem{coleSearchingDynamic2006}
Richard Cole and Lee-Ad Gottlieb.
\newblock Searching dynamic point sets in spaces with bounded doubling dimension.
\newblock In {\em Proceedings of the Thirty-Eighth Annual {{ACM}} Symposium on {{Theory}} of {{Computing}} (STOC '06)}, pages 574--583, 2006.
\newblock \href {https://doi.org/10.1145/1132516.1132599} {\path{doi:10.1145/1132516.1132599}}.

\bibitem{dey2014computing}
Tamal~K Dey, Fengtao Fan, and Yusu Wang.
\newblock Computing topological persistence for simplicial maps.
\newblock In {\em Proceedings of the thirtieth annual symposium on Computational geometry (SoCG '14)}, pages 345--354, 2014.
\newblock \href {https://doi.org/10.1145/2582112.2582165} {\path{doi:10.1145/2582112.2582165}}.

\bibitem{dey2019simba}
Tamal~K Dey, Dayu Shi, and Yusu Wang.
\newblock Simba: An efficient tool for approximating {Rips}-filtration persistence via simplicial batch collapse.
\newblock {\em Journal of Experimental Algorithmics (JEA)}, 24:1--16, 2019.
\newblock \href {https://doi.org/10.1145/3284360} {\path{doi:10.1145/3284360}}.

\bibitem{eppsteinOptimalSpanners2022a}
David Eppstein and Hadi Khodabandeh.
\newblock Optimal {{Spanners}} for {{Unit Ball Graphs}} in {{Doubling Metrics}}, 2022.
\newblock \textit{arXiv preprint}.
\newblock \href {https://arxiv.org/abs/2106.15234} {\path{arXiv:2106.15234}}.

\bibitem{GKL03}
A.~Gupta, R.~Krauthgamer, and J.R. Lee.
\newblock Bounded geometries, fractals, and low-distortion embeddings.
\newblock In {\em 44th {{Annual IEEE Symposium}} on {{Foundations}} of {{Computer Science}}}, pages 534--543, 2003.
\newblock \href {https://doi.org/10.1109/SFCS.2003.1238226} {\path{doi:10.1109/SFCS.2003.1238226}}.

\bibitem{har-peledFastConstruction2005}
Sariel {Har-Peled} and Manor Mendel.
\newblock Fast construction of nets in low dimensional metrics, and their applications.
\newblock In {\em Proceedings of the Twenty-First Annual Symposium on {{Computational}} Geometry {(SoCG '05)}}, pages 150--158, 2005.
\newblock \href {https://doi.org/10.1145/1064092.1064117} {\path{doi:10.1145/1064092.1064117}}.

\bibitem{Hei01}
Juha Heinonen.
\newblock {\em Lectures on {{Analysis}} on {{Metric Spaces}}}.
\newblock Universitext. Springer New York, 2001.
\newblock \href {https://doi.org/10.1007/978-1-4613-0131-8} {\path{doi:10.1007/978-1-4613-0131-8}}.

\bibitem{hellmerDensity24}
Niklas Hellmer and Jan Spaliński.
\newblock Density {Sensitive} {Bifiltered} {Dowker} {Complexes} via {Total} {Weight}, 2024.
\newblock \textit{arXiv preprint}.
\newblock \href {https://arxiv.org/abs/2405.15592} {\path{arXiv:2405.15592}}.

\bibitem{kerber2019barcodes}
Michael Kerber and Hannah Schreiber.
\newblock Barcodes of towers and a streaming algorithm for persistent homology.
\newblock {\em Discrete \& {Computational} {Geometry}}, 61:852--879, 2018.
\newblock \href {https://doi.org/10.1007/s00454-018-0030-0} {\path{doi:10.1007/s00454-018-0030-0}}.

\bibitem{landi2018rank}
Claudia Landi.
\newblock The rank invariant stability via interleavings.
\newblock In Erin~Wolf Chambers, Brittany~Terese Fasy, and Lori Ziegelmeier, editors, {\em Research in Computational Topology}, pages 1--10. Springer International Publishing, 2018.
\newblock \href {https://doi.org/10.1007/978-3-319-89593-2_1} {\path{doi:10.1007/978-3-319-89593-2_1}}.

\bibitem{lesnickTheoryInterleaving2015}
Michael Lesnick.
\newblock The {Theory} of the {Interleaving} {Distance} on {Multidimensional} {Persistence} {Modules}.
\newblock {\em Foundations of Computational Mathematics}, 15(3):613--650, 2015.
\newblock \href {https://doi.org/10.1007/s10208-015-9255-y} {\path{doi:10.1007/s10208-015-9255-y}}.

\bibitem{lesnickNerveModels2024}
Michael Lesnick and Ken McCabe.
\newblock Nerve {{Models}} of {{Subdivision Bifiltrations}}, 2024.
\newblock \textit{arXiv preprint}.
\newblock \href {https://arxiv.org/abs/2406.07679} {\path{arXiv:2406.07679}}.

\bibitem{lesnickInteractive15}
Michael Lesnick and Matthew Wright.
\newblock Interactive {Visualization} of 2-{D} {Persistence} {Modules}, 2015.
\newblock \textit{arXiv preprint}.
\newblock \href {https://arxiv.org/abs/1512.00180} {\path{arXiv:1512.00180}}.

\bibitem{rolle2020stable}
Alexander Rolle and Luis Scoccola.
\newblock Stable and consistent density-based clustering via multiparameter persistence, 2023.
\newblock \textit{arXiv preprint}.
\newblock \href {https://arxiv.org/abs/2005.09048} {\path{arXiv:2005.09048}}.

\bibitem{Scoccola2023}
Luis Scoccola and Alexander Rolle.
\newblock Persistable: persistent and stable clustering.
\newblock {\em Journal of Open Source Software}, 8(83):5022, 2023.
\newblock \href {https://doi.org/10.21105/joss.05022} {\path{doi:10.21105/joss.05022}}.

\bibitem{sheehyMulticoverNerve2012a}
Donald~R. Sheehy.
\newblock A {Multicover} {Nerve} for {Geometric} {Inference}.
\newblock In {\em 24th {Canadian} {Conference} on {Computational} {Geometry} ({CCCG} 2012)}, pages 309--314, 2012.
\newblock URL: \url{http://2012.cccg.ca/e-proceedings.pdf}.

\bibitem{sheehyLinearSizeApproximations2013a}
Donald~R. Sheehy.
\newblock Linear-{{Size Approximations}} to the {{Vietoris}}--{{Rips Filtration}}.
\newblock {\em Discrete \& Computational Geometry}, 49(4):778--796, 2013.
\newblock \href {https://doi.org/10.1007/s00454-013-9513-1} {\path{doi:10.1007/s00454-013-9513-1}}.

\bibitem{sheehy2021sparse}
Donald~R. Sheehy.
\newblock A sparse {Delaunay} filtration.
\newblock In {\em 37th International Symposium on Computational Geometry (SoCG 2021)}, 2021.
\newblock \href {https://doi.org/10.4230/LIPIcs.SoCG.2021.58} {\path{doi:10.4230/LIPIcs.SoCG.2021.58}}.

\bibitem{smidWeakGap2009a}
Michiel Smid.
\newblock The {{Weak Gap Property}} in {{Metric Spaces}} of {{Bounded Doubling Dimension}}.
\newblock In {\em Efficient {{Algorithms}}: {{Essays Dedicated}} to {{Kurt Mehlhorn}} on the {{Occasion}} of {{His}} 60th {{Birthday}}}, pages 275--289. Springer, 2009.
\newblock \href {https://doi.org/10.1007/978-3-642-03456-5_19} {\path{doi:10.1007/978-3-642-03456-5_19}}.

\bibitem{talwar2004bypassing}
Kunal Talwar.
\newblock Bypassing the embedding: algorithms for low dimensional metrics.
\newblock In {\em Proceedings of the thirty-sixth annual ACM symposium on Theory of computing (STOC '04)}, pages 281--290, 2004.
\newblock \href {https://doi.org/10.1145/1007352.1007399} {\path{doi:10.1145/1007352.1007399}}.

\end{thebibliography}

\end{document}